\documentclass{article}
\usepackage{graphicx} 
\usepackage{url,amsmath,amsfonts,amssymb, xcolor}
\usepackage{float}
\usepackage{diagbox}
\usepackage{algorithm}
\usepackage{algpseudocode}
\usepackage{caption}
\usepackage{subcaption}
\usepackage{tikz}

\newtheorem{theorem}{Theorem}
\newtheorem{problem}{Problem}
\newtheorem{lemma}[theorem]{Lemma}

\newtheorem{remark}{Remark}
\newenvironment{proof}{{\bf Proof.}}{\hfill$\square$}
\newenvironment{proofof}[1]{{\bf Proof #1.}}{\hfill$\square$}

\newcommand{\dual}[2]{\left<#1,#2\right>}

\newcommand{\I}{\mathrm I}
\newcommand{\foralls}{\forall\,}
\newcommand{\diff}{\kappa}
\newcommand{\lm}{\lambda}

\newcommand{\supp}{\mathrm{supp}\,}
\newcommand{\Cf}{\mathrm{Cf}\,}

\newcommand{\RR}{\mathbb{R}}
\newcommand\Strut{\rule{0pt}{2.6ex}\rule[-0.9ex]{0pt}{0pt}}

\title{A robust and time-parallel preconditioner for parabolic reconstruction problems using Isogeometric Analysis}
\author{
Kent-Andre Mardal,\footnote{kent-and@simula.no, Department of Mathematics, University of Oslo, Postboks 1053, Blindern, Oslo 0316, Norway;
KAM was supported by the
 Research Council of Norway
through the grants 300305 and 301013.
}
{}
Jarle Sogn,\footnote{jarle.sogn@gmail.com, Department of Mathematics, University of Oslo, Postboks 1053, Blindern, Oslo 0316, Norway}
{}
and Stefan Takacs\footnote{stefan.takacs@numa.uni-linz.ac.at, Institute of Numerical Mathematics, Johannes Kepler University Linz, Altenberger Str. 69, 4040 Linz, Austria; this work was supported by the Austrian Science Fund (FWF): 10.55776/P33956
}
}
\date{July 2024}

\begin{document}

\maketitle

\begin{abstract}
    We consider a PDE-constrained optimization problem of tracking type with parabolic state equation.
    The solution to the problem is characterized by the Karush–Kuhn–Tucker (KKT) system, which we formulate using a strong variational formulation of the state equation and a super weak formulation of the adjoined state equation. This allows us to propose a preconditioner that is robust both in the regularization and the diffusion parameter. In order to discretize the problem, we use Isogeometric Analysis since it allows the construction of sufficiently smooth basis functions effortlessly.
    To realize the preconditioner, one has to solve a problem over the whole space time cylinder that is elliptic with respect to certain non-standard norms.
    Using a fast diagonalization approach in time, we reformulate the problem as a collection of elliptic problems in space only. These problems are not only smaller, but our approach also allows to solve them in a time-parallel way.
    We show the efficiency of the preconditioner by rigorous analysis and illustrate it with numerical experiments.
\end{abstract}

\section{Introduction}
In this paper, we consider the following PDE-constrained optimization problem of tracking type with a parabolic state equation. Given a bounded domain $\Omega \subset \mathbb R^d$, $d=2,3$, with Lipschitz boundary $\partial \Omega$, the space-time cylinder $Q_T = (0,T]\times \Omega$ with time horizon $T>0$, an observation domain $q_t \subseteq Q_T$ and a desired or observed state $y_d$, a diffusion parameter $\kappa>0$, and a regularization parameter $\alpha>0$,
find the state $y$ and the control $u$ such that
\begin{equation}
\label{problem}
\min_{y,u}
\frac 12 \|y-y_d\|_{L^2(q_t)}^2
+ \frac{\alpha}{2} \|u\|_{L^2(Q_T)}^2
\ \mbox{subject to} \ \frac{\partial y}{\partial t} - \kappa \Delta y = u
\end{equation}
with suitable boundary and initial conditions.
Our main motivation for this problem class is the reconstruction of parabolic evolution processes based on images or measurements $y_d$; this might be usually a distributed observation on parts of the time interval ($q_t = \mathcal T_{\mathcal O} \times \Omega$ for some $\mathcal T_{\mathcal O} \subset (0,T]$) or a limited observation of the whole time interval ($q_t = (0,T] \times \Omega_{\mathcal O}$ for some $\Omega_{\mathcal O} \subset \Omega$). Alternatively, we might be interested in the case that the observation domain only consists of discrete points in time or only on the boundary of $\Omega$, which will be addressed as well. Applications of such problems are for example medical imaging such as the so-called glymphatic magnetic resonance imaging (gMRI) where MRI contrast is imaged at multiple time points to show the fluid flow and clearance throughout the brain~\cite{ringstad2018brain,valnes2020apparent,vinje2023human}.
We are interested in a preconditioner that this robust in the model parameters (like $\alpha$ and $\kappa$), the length of the time steps and the grid sizes.

The solution to the problem is characterized by the corresponding optimality or Karush–Kuhn–Tucker (KKT) system. After a discretization using the Galerkin method, one obtains a problem for the whole space-time cylinder, so a three or four dimensional problem. For the solution of such a problem, iterative solvers with appropriate preconditioners are warranted.
We are interested in preconditioners that are robust in the regularization parameter.
Most of the earlier approaches to preconditioners that are provably robust in the regularization parameter are restricted to the case of full observation, i.e., $q_t=Q_T$ in our notation. Such robustness was archived for elliptic problems in \cite{SchZul07,zulehner2011nonstandard,PeaWat12} and others, later extended to parabolic problems in \cite{kolmbauer2012robust, liang2021robust,pearson2012regularization} and others. For these approaches, a direct extension to the case of limited observation is not known.

In \cite{MarNieNor17}, it was shown that robustness in the regularization parameter could be guaranteed also for the case of limited observation if the state equation is represented by a strong formulation and the adjoined state equation is formulated in a super weak formulation. Considering the Poisson equation, this means that the state variable should be two times (weakly) differentiable, while no differentiability requirements are formulated for the control and the adjoined state. A conforming discretization to such a problem needs to be $C^1$-smooth. In the context of Finite Element Methods with triangular meshes, this can be archived, e.g. using the Argyris element \cite{argyris} in two dimensions, leading to 21 degrees of freedom per element. In Isogeometric Analysis, cf.~\cite{HugCotBaz05,VeiBufSanVaz14}, the construction of arbitrarily smooth basis functions can be done effortlessly. The authors have worked out the details in \cite{mardal2022robust}, where they have shown that the approach also works for the convection-diffusion-reaction equation, even in the convection dominated case.
Both, in \cite{MarNieNor17, mardal2022robust}, the theory was developed from the continuous perspective and its discretization was subsequently addressed.

That approach can be extended to the parabolic problem considered in this paper. Here, the state variable needs to be two times weakly differentiable in space and once weakly differentiable in time. Starting from the continuous formulation, we formulate the optimality system with strong variational formulation of the state equation and super weak formulation of the adjoined state equation and discuss the setup of a robust preconditioner for the continuous formulation. That preconditioner can be interpreted as an operator preconditioner; for its analysis, we need to show an inf-sup condition, specifically a discrete inf-sup condition. For the analysis, we verify the corresponding discrete inf-sup condition. Unlike our previous article \cite{mardal2022robust}, where we stated the discrete inf-sup condition only for trivial geometries, we now show that the discrete inf-sup condition holds for all (sufficiently smooth) geometries, provided that the grid is fine enough.

In order to realize the application of preconditioner, an elliptic problem has to be solved over the whole space-time cylinder. Since the observation domain $q_t$ is often either distributed in space or distributed in time, the matrix representing the preconditioner has a Kronecker-product structure. So, it can be realized by applying the fast diagonalization technique, cf. \cite{sangalli2016isogeometric}, to the time variable. Instead of one elliptic problem over the whole space-time cylinder, one has to solve several elliptic problems in space only. Specifically, the number of these problems is as large as the number of time steps. Since these elliptic problems are independent of each other, they could be solved using in a time-parallel or, maybe more accurately, in a frequency-parallel way.
As such the solver is similar to the various parallel in
time schemes for solvers of
parabolic forward problems that have been proposed, such as
for example~\cite{farrell2021irksome,   leveque2023fast,loli2020efficient, mardal2007order, staff2005stability},
and is easy to implement in
a block Jacobi or domain-decomposition framework in time.

This paper is organized as follows. In Section~\ref{sec:problem}, we state the model problem concisely, discuss the proper function spaces and show the well-posedness of the problem on the continuous level. The discretization of the problem and the optimality of the preconditioner are shown in Section~\ref{sec:iga}; some of the auxiliary results for that analysis are to be found in the Appendix. The application of fast diagonalization in time is then discussed in Section~\ref{sec:fd}. In Section~\ref{sec:numeric}, we give numerical experiments illustrating our results. We close with conclusions in Section~\ref{sec:conclusions}.

\section{Problem formulation and preconditioner on continuous level}
\label{sec:problem}

In this section, we give a concise definition of the model problem. Let $\Omega\subset{\mathbb{R}^d}$ be a bounded open domain with a Lipschitz boundary $\partial \Omega$ and let $Q_T:= (0,T)\times\Omega $ be the space-time cylinder, where $T>0$ is the end of the considered time interval and $q_t\subseteq Q_T$ be an observation domain. The optimization problem reads as follows; the precise regularity assumptions on data and solutions are stated below.
\begin{problem}\label{prob1}
Given
a desired state $y_d$,
an initial state $y_0$,
boundary data $g$,
a source $f$,
and parameters $\kappa , \alpha > 0$, minimize
\begin{equation}
\label{eq:costFunc}
 J(y,u) = \frac12 \|y-y_d\|_{L^2(q_t)}^2 + \frac{\alpha}{2} \|u\|^2_{L^2(Q_T)}
\end{equation}
subject to the second order parabolic equation
\begin{align}
\label{eq:PDEStrong}
\begin{aligned}
  \partial_t y(t,x) - \diff\Delta y(t,x) &= f(t,x) - u(t,x)
  &&\quad \text{in}&&\quad (0,T]\times \Omega,\\
  y(t,x) &= g(t,x)
  &&\quad \text{in}&&\quad (0,T] \times \partial\Omega,\\
  y(0,x) &= y_0(x)
  &&\quad \text{in}&&\quad\Omega,
  \end{aligned}
\end{align}
where $\partial_t y:= \frac{\partial y}{\partial t}$ is the partial derivative with respect to time and $\Delta$ is the Laplace operator with respect to space.
\end{problem}

We assume that the initial and boundary conditions are \emph{compatible,} i.e., that there is a sufficiently smooth function $Q_T \to \RR$ that satisfies both conditions. Since this allows homogenization, we assume without loss of generality that $y_0=0$ and $g=0$.

We write the constraint in a strong variational form. For this purpose, we introduce the following function spaces and notation. $L^2(\Omega)$ denotes the standard Lebesgue space of square-integrable functions $\Omega\to \RR$. Analogously, $L^2( (0,T); V)$ denotes the Bochner space of square-integrable functions, mapping from the time interval into the Hilbert space $V$. $H^m(\Omega)$ and $H^m( (0,T); V)$ denote the Sobolev spaces of $m$ times weakly differentiable functions $\Omega\to \RR$ and $(0,T)\to V$, respectively. $H^1_0(\Omega) \subset H^1(\Omega)$ is the subspace of functions vanishing on the boundary, $H^1_{0,*}((0,T);V) \subset H^1((0,T);V)$ is the subspace of functions vanishing for $t=0$, but not necessarily for $t=T$. These function spaces are equipped with the standard norms $\|\cdot\|_{L^2}$ and $\|\cdot\|_{H^1}$. Moreover, let
\[
    H(\Delta, \Omega) :=\lbrace v\in L^2(\Omega) : \Delta v\in L^2(\Omega) \rbrace
\]
be the space of functions whose weak Laplacian is in $L^2(\Omega)$, whose norm is given by $\|v\|_{H(\Delta, \Omega)}^2 := \|v\|_{L^2(\Omega)} ^2+ \|\Delta v\|_{L^2(\Omega)}^2$. For a Hilbert space $H$, we use the apostrophe to denote its dual space $H'$.

As function space to formulate the state equation in a strong variational formulation, we choose
\[
        Y :=
        L^2((0,T);H(\Delta, \Omega)\cap H^1_0(\Omega))
        \cap
        H^1_{0,*}((0,T); L^2(\Omega)),
\]
which is a complete space with respect to the norm
\[
        \|y\|_{Y}^2 := \|\partial_t y - \kappa \Delta y \|_{L^2(Q_T)} ^2,
\]
cf. \cite[Chapter 3, Theorem 2.1]{ladyzhenskaya2013boundary}.
We search for the state $y$ in the function space $Y$ and for the control $u$ in $U:=L^2(Q_T)$. The constraint equation is multiplied with a test function $\lm \in U$ and we integrate over the space-time domain.
We obtain the following variational formulation
\begin{equation*}
    (\partial_t y-\diff\Delta y, \lm)_{L^2(Q_T)}  = (f-u,\lm)_{L^2(Q_T)}  \quad \foralls \lm \in U.
\end{equation*}
The Lagrangian functional associated to this problem is
\[
\mathcal{L}(y,u,\lm) := \frac{1}{2}\|y-y_d\|^2_{L^2(q_t)} + \frac{\alpha}{2} \|u\|^2_{L^2(Q_T)} +(\partial_t y-\diff\Delta y-f+u,\lm)_{L^2(Q_T)}.
\]
The solution to Problem~\ref{prob1} is a critical point of the Lagrangian functional, which is characterized by the Karush-Kuhn-Tucker system, so it is equivalent to the following problem.
\begin{problem}
\label{prob:variational}
    Find $(y,u,\lambda)\in Y\times U\times U$ such that
    \begin{align*}
       (y,\tilde{y})_{L^2(q_t)}+ (\lm,\partial_t \tilde{y}-\diff\Delta \tilde{y})_{L^2(Q_T)} &=(y_d,\tilde{y})_{L^2(q_t)}  \quad &\foralls \tilde{y} \in Y,\\
   \alpha(u,\tilde{u})_{L^2(Q_T)} + (\lm,\tilde{u})_{L^2(Q_T)} &= 0 \quad  &\foralls \tilde{u} \in U, \\
   (\partial_t y-\diff\Delta y,\tilde{\lm})_{L^2(Q_T)} +  (u,\tilde{\lm})_{L^2(Q_T)}  &= (f,\tilde{\lm})_{L^2(Q_T)} \quad  &\foralls \tilde{\lm} \in U.
    \end{align*}
\end{problem}
We write Problem~\ref{prob:variational} in operator notation as $\mathcal{A}\mathbf{x} = \mathbf{b}$, where
\begin{equation}
\label{eq:3by3}
\mathcal{A} = \begin{pmatrix}
  M_{q_t} & 0 & L'\\
  0& \alpha M & M\\
  L& M & 0
\end{pmatrix},\quad \mathbf{x}=
\begin{pmatrix}
y\\ u\\ \lm
\end{pmatrix},\quad \mathbf{b}=
\begin{pmatrix}
M_{q_t}y_d\\ 0\\ M f
\end{pmatrix}.
\end{equation}
Here, $M: U\rightarrow U'$
and $M_{q_t}: U\rightarrow U'$
represent the $L^2$-inner products
\[
    \dual{Mu}{\lm} = (u,\lm)_{L^2(Q_T)}
        \quad\mbox{and}\quad
    \dual{M_{q_t} u}{\lm} = (u,\lm)_{L^2(q_T)}
\]
and $L: Y\rightarrow U'$ represents the differential operator
\[
        \dual{Ly}{\lm} = (\partial_t y-\diff\Delta y,\lm)_{L^2(Q_T)}.
\]
By reordering the rows and columns in~\eqref{eq:3by3}, one obtains the double saddle point operator $\mathcal A$, for which the Schur complement preconditioners $\mathcal B_0$ from~\cite{MarNieNor17,SogZul18} are available, which read as follows
\begin{equation}\label{eq:schur:reordered}
\tilde A=
\begin{pmatrix}
  \alpha M & M&0\\
  M & 0&L\\
  0 & L'&M_{q_t}
\end{pmatrix}
\quad\mbox{and}\quad
\widetilde{\mathcal{B}}_0 :=
\begin{pmatrix}
  \alpha M & 0 & 0\\
  0 & \frac{1}{\alpha}M & 0\\
  0 & 0 & M_{q_t} + \alpha L'M^{-1}L
\end{pmatrix}.
\end{equation}
The preconditioner for $\mathcal A$ as in~\eqref{eq:3by3} is obtained by reordering again, which yields
\begin{equation*}
\mathcal{B}_0 :=
    \begin{pmatrix}
  M_{q_t} + \alpha L'M^{-1}L & 0 & 0\\
  0& \alpha M & 0&\\
  0& 0 & \frac{1}{\alpha}M
\end{pmatrix}.
\end{equation*}
By examining the operator $L' M^{-1}L$, we see that
\begin{equation}\label{eq:as:norm}
\begin{aligned}
    \dual{L' M^{-1}L y}{y} &= \sup_{\lm\in U}\frac{\dual{L y}{\lm}^2}{\|\lm\|^2_{L^2(Q_T)}}
    =\sup_{\lm\in U}\frac{(\partial_t y-\diff\Delta y,\lm)^2_{L^2(Q_T)}}{\|\lm\|^2_{L^2(Q_T)}} \\
    &= \|\partial_t y-\diff\Delta y\|^2_{L^2(Q_T)} = \|y\|_Y^2.
\end{aligned}
\end{equation}
The third equality follows from the Cauchy-Schwarz inequality and
\begin{equation}\label{eq:inclusion}
        (\partial_t - \kappa \Delta ) Y \subseteq U,
\end{equation}
which means that $\lambda:=\partial_t y-\Delta y \in U$ for all $y\in Y$.

\begin{lemma}\label{lem:2}
    Problem~\ref{prob:variational} is uniquely solvable; moreover,
    the condition number of the preconditioned system $\mathcal{B}_0^{-1}\mathcal{A}$ is bounded as follows
    \[
        \kappa \left(\mathcal{B}_0^{-1}\mathcal{A}\right) \leq \frac{\cos(\pi/7)}{\sin(\pi/14)} \approx 4.05.
    \]
\end{lemma}
\begin{proof}
We use the results from \cite{SogZul18}. By reordering rows and columns in $\mathcal A$ and $\mathcal B_0$, we obtain the operators from~\eqref{eq:schur:reordered}, which have the form as in that paper.
Then, the well-posedness follows from \cite[Lemma 2.1]{SogZul18} by using Lemma~\ref{lemma:YZeq}. The condition number bound can be found in \cite[Corollary 2.4]{SogZul18}.
\end{proof}

\begin{remark}[2--by--2 formulation]
The 3--by--3 formulation \eqref{eq:3by3} can be reduced to a 2--by--2 formulation by eliminating the control using $u=-\tfrac1\alpha \lambda$.
The reduced system and preconditioner are
\begin{equation}
\label{eq:2by2}
\begin{pmatrix}
  M_{q_t}  & L'\\
  L& -\tfrac{1}{\alpha}M &
\end{pmatrix}
\begin{pmatrix}
y\\ \lm
\end{pmatrix}=
\begin{pmatrix}
{M}_{q_t}y_d\\  M f
\end{pmatrix}
\quad \text{and} \quad
\begin{pmatrix}
  M_{q_t} + \alpha L'M^{-1}L  & 0\\
  0& \tfrac{1}{\alpha}M &
\end{pmatrix}.
\end{equation}
Analogously to Lemma~\ref{lem:2}, we can show well-posedness also for the reduced problem. Here, the condition number bound is $\frac{ \cos ( \pi / 5 ) }{ \sin ( \pi / 10 ) } \approx 2.62$.
\end{remark}

In order to realize the preconditioner for either the 3--by--3 formulation or the 2--by--2 formulation, we have to solve linear systems of the form, given $w\in U$, find $y\in Y$ with
\[
        ( M_{q_t} + \alpha L' M^{-1} L) y = w,
\]
which warrants a closer look on the corresponding operator. The following lemma is helpful for this task.
\begin{lemma}
\label{lemma:YZeq}
We have
\[
\tfrac 12 \|y\|^2_Y  \leq  \|\partial_t y\|^2_{L^2(Q_T)} + \kappa^2 \|\Delta y\|^2_{L^2(Q_T)}  \leq \|y\|^2_Y
\quad \foralls y\in Y.
\]
\end{lemma}
\begin{proof}
    Using the triangle inequality and Young's inequality, we obtain
    \begin{align*}
        \|y\|_{Y}^2 =
         \|\partial_t y - \kappa \Delta y \|^2_{L^2(Q_T)} &\leq 2\|\partial_t y\|^2_{L^2(Q_T)} + 2\kappa^2 \|\Delta y\|^2_{L^2(Q_T)},
    \end{align*}
    which shows the first inequality. Using the binomial expansion and integration by parts in space, the fundamental theorem of calculus, and $y(0)=0$, we obtain
    \begin{align*}
        &\|y\|^2_{Y} - \|\partial_t y\|^2_{L^2(Q_T)} - \kappa^2 \|\Delta y\|^2_{L^2(Q_T)}
         =         2\kappa (\partial_t y, - \Delta y)_{L^2(Q_T)}\\
        & \quad =  2\kappa(\nabla\partial_t y,  \nabla y)_{L^2(Q_T)}
        = \tfrac{\mathrm d}{\mathrm dt} \kappa (\nabla y,  \nabla y)_{L^2(Q_T)}
        = \kappa \|\nabla y(T) \|^2_{L^2(\Omega)} \ge 0,
    \end{align*}
    i.e., the second inequality.
\end{proof}

Let $B,C:Y\to Y'$ be operators representing the biharmonic problem in space and the harmonic problem in time, respectively, i.e.,
\[
        \dual{B y}{\tilde y} := (\Delta y,\Delta \tilde{y})_{L^2(Q_T)},
        \qquad
        \dual{C y}{\tilde y} := (\partial_t y,\partial_t \tilde{y})_{L^2(Q_T)},
\]
Using~\eqref{eq:as:norm} and Lemma~\ref{lemma:YZeq}, we immediately obtain that $\mathcal B_0$ is spectrally equivalent to
\begin{equation*}
\mathcal{B} :=
    \begin{pmatrix}
  M_{q_t} + \alpha \kappa^2 B + \alpha C  & 0 & 0\\
  0& \alpha M & 0&\\
  0& 0 & \frac{1}{\alpha}M
\end{pmatrix}.
\end{equation*}
So, using Lemma~\ref{lem:2}, we obtain further as follows.
\begin{lemma}\label{lem:3}
    The condition number of the preconditioned system $\mathcal{B}^{-1}\mathcal{A}$ is bounded as follows
    \[
        \kappa \left(\mathcal{B}^{-1}\mathcal{A}\right) \leq 2\frac{\cos(\pi/7)}{\sin(\pi/14)} \approx 8.10.
    \]
\end{lemma}

It is worth noting that we did not assume any tensor-product structure of the observation domain $q_t$. Before we discuss the discretization, we want to remark on possible extensions.
\begin{remark}[Neumann and Robin conditions]\label{rem:neu}
        The approach discussed in this paper can be extended to Neumann and Robin boundary conditions.
        Since we consider a strong formulation of the state equation, the (homogenized) Neumann or Robin conditions have to be incorporated in the space $Y$ as well. This is no problem since Neumann traces are well-defined on $H(\Delta,\Omega)$.
        Specifically, $\gamma_n : H(\Delta,\Omega) \to H^{-1/2}(\partial\Omega)$ with $\langle \gamma_n v, w\rangle = (\Delta v,w)_{L^2(\Omega)}+(\nabla v,\nabla w)_{L^2(\Omega)}$ defines a Neumann trace, i.e., it satisfies
        \[
                \gamma_n w = \frac{\partial w}{\partial n}\big|_{\partial \Omega}
                \quad \forall w \in C^\infty (\Omega).
        \]
\end{remark}
The analysis is not limited to $L^2$ tracking functionals. Indeed, we can replace $\tfrac12 \|y-y_d\|_{L^2(q_t)}^2$ by $\tfrac12 \dual{ M_{q_t} (y-y_d)}{y-y_d}$ for any linear and continuous operator $M_{q_t}:\tilde Y\to \tilde Y'$, where $\tilde Y$ is a Hilbert space with $Y\subseteq \tilde Y$ and $y_d \in \tilde Y$. This allows the extension to the following cases.
\begin{remark}[Boundary observation]\label{rem:bo}
        The analysis also holds for the tracking functional
        \[
            \hat J(y,u):=
            \frac 12 \|y - y_d\|_{L^2((0,T)\times \partial \Omega)}^2
            + \frac \alpha2 \|u\|_{L^2(Q_T)}^2
        \]
        for given $y_d \in L^2((0,T)\times\partial \Omega)$. Certainly, this is only of interest if Neumann or Robin conditions are prescribed (Remark~\ref{rem:neu}).
        $M_{q_t}$ is continuous on $\hat Y:=L^2((0,T)\times \partial \Omega)$, which is a subspace of $Y$ due to standard trace theorems for $H^1(\Omega)$:
        \[
            \|y\|_{\tilde Y}^2 = \|y\|_{L^2((0,T)\times \partial \Omega)}^2
            = \|y\|_{L^2((0,T);L^2(\partial \Omega))}^2
            \le c \|y\|_{L^2((0,T),H^1(\Omega))}^2\le c \|y\|_Y^2.
        \]
\end{remark}
\begin{remark}[Observation on discrete points in time]\label{rem:ti}
        The analysis also holds for the tracking functional
        \[
            \tilde J(y,u):=
            \frac 12 \sum_{i=1}^N \|y(t_i) - y_d(t_i)\|_{L^2(\Omega)}^2
            + \frac \alpha2 \|u\|_{L^2(Q_T)}^2
        \]
        for given $t_i \in [0,T]$ and $y_d(t_i) \in L^2(\Omega)$ for $i=1,\ldots,N$.
%
        $M_{q_t}$ is continuous on the Bochner space  $\hat Y:=H^1((0,T),L^2(\Omega))$, which is a subspace of $Y$ due to standard trace theorems:
        \[
            \dual{M_{q_t} y}{y} = \sum_{i=1}^N \|y(t_i) \|_{L^2(\Omega)}^2
            \le c \|y\|_{\tilde Y}^2 = \|y\|_{H^1((0,T),L^2(\Omega))}^2\le c \|y\|_Y^2.
        \]
\end{remark}

\section{Discretization with Isogeometric Analysis}
\label{sec:iga}
We discretize the optimality system with a space-time isogeometric discretization, which we introduce briefly. For more information on Isogeometric Analysis, we refer to the literature, like~\cite{HugCotBaz05,HugCotBaz09} and references therein. We assume that the domain $\Omega\subset \RR^d$ is parameterized by a geometry function
\begin{equation}
\mathbf{G}: \widehat \Omega \rightarrow \Omega = \mathbf{G}(\widehat\Omega),
\end{equation}
defined on the parameter domain  $\widehat \Omega := (0,1)^d$. We assume that the geometry function is sufficiently regular, specifically we assume
\begin{equation}\label{eq:ass:geo}
    \mathbf G\in W^3_\infty(\widehat\Omega)
    \quad \text{and} \quad
    |\det \nabla\mathbf{G}| \ge c_{\mathbf G} >0, \quad c_{\mathbf G}\in \mathbb R,
\end{equation}
where $\nabla\mathbf{G}$ denotes the Jacobian.
To obtain a parametrization of the space-time cylinder $Q_T= (0,T)\times \Omega$, we use
\[
    \mathbf{G}_T: \widehat Q_T := (0,T) \times \widehat\Omega \rightarrow Q_T,
    \quad\text{given by}\quad
    \mathbf{G}_T(t,\widehat x) := (t, \mathbf G(\widehat x)).
\]

For discretization, we use tensor-product B-splines, which are defined on the parameter domain $\widehat Q_T$.
For any grid $Z = (z_0,z_1,z_2,\cdots, z_N)$ with $z_i<z_{i+1}$, any integers $p>k\ge0$, the associated spline space $S_{p,k}(Z)$ consists of all functions in $C^{k}(z_0,z_N)$ whose restriction to any interval $[z_i,z_{i+1}]$ between two consecutive breakpoints is a polynomial function of degree $p$. We can analogously define spline spaces for $p>k=-1$; these functions might be discontinuous, on each of the intervals $[z_0,z_1]$ and $(z_i,z_{i+1}]$, $i=1,\ldots, N-1$, they are assumed to be polynomial. The largest span between breakpoints is the grid size $h_Z:=\max_i z_i-z_{i-1}$; analogously, we define $h_{Z,\min}:=\min_i z_i-z_{i-1}$.

To obtain a function space over $\widehat \Omega$, we use tensor-product splines over some grid $\mathbf Z_x = (Z_1,\cdots,Z_d)$, which we denote by $S_{p_x,k_x,\mathbf Z_x}$. Analogously, we obtain a function space over $[0,T]$ by  introducing a grid $Z_t$ for the time domain and by utilizing the corresponding spline space $S_{p_t,k_t,Z_t}$ as defined above. To obtain a function space over the whole space-time cylinder, we again use tensor-product splines
\[
        S_{p_t,k_t,Z_t} \otimes S_{p_x,k_x,\mathbf Z_x}
\]
with parameters still to be defined.

Now, we can specify the spaces $Y_h \subset Y$ and $U_h \subset U$ that we use to discretize the problem.
By choosing a spline space of functions that are continuous in time direction and continuously differentiable in space direction, we obtain a conforming space, i.e., a subspace of $Y$. Accordingly, we choose
\begin{equation}
    \label{eq:Ydisc}
    \widehat Y_h:= \big\{ \widehat y_h :  \widehat y_h \in S_{p_t,k_t,Z_t} \otimes S_{p_x,k_x,\mathbf Z_x}, \; \widehat y_h\big|_{t=0} = 0\mbox{ and } \widehat y_h\big|_{(0,T]\times \partial \widehat \Omega} = 0 \big\}
\end{equation}
for some $p_t>k_t\ge0$ and $p_x>k_x\ge1$.
The space $Y_h \subset Y$ is then defined via the pull-back principle, i.e.,
\[
    Y_h:=\widehat Y_h \circ \mathbf G_T^{-1}.
\]
Since we will need this in the next section, we want to stress that we consider also a corresponding tensor-product basis for these spaces. So, let $[ \varphi_i]_{i=0}^{N_t-1}$ and $[\widehat \phi_j]_{j=0}^{N_x-1}$ be the bases of $S_{p_t,k_t,Z_t}$ and $S_{p_x,k_x,\mathbf Z_x}$, respectively. Then, we choose $[\widehat \psi_n]_{n=0}^{N_tN_x-1}$ with $\widehat \psi_{iN_x+j}(t,x)= \varphi_i(t)\;\widehat \phi_j(x)$ as basis for $\widehat Y_h$. According to the pull-back principle, $[\psi_n]_{n=0}^{N_tN_x-1}$ with $\psi_{iN_x+j}(t,x)=\widehat \psi_{iN_x+j}(\mathbf G_T^{-1}(t,x)) = \widehat \psi_{iN_x+j}(t, \mathbf G^{-1}(x)) = \varphi_i(t)\;\widehat \phi_j(\mathbf G^{-1}(x))$ is the basis for $Y_h$. With the choice $\phi_j(x):=\widehat \phi_j(\mathbf G^{-1}(x))$, we see that the basis functions have the following form:
\begin{equation}\label{eq:tps}
        \psi_{iN_x+j}(t,x)=\varphi_i(t)\;\phi_j(x).
\end{equation}
Concerning the choice of the function space for the control $u_h$ and the Lagrange multiplier $\lambda_h$, we follow~\eqref{eq:inclusion} and choose
\begin{equation}
\label{eq:discCon}
\widehat U_h:= S_{p_t,k_t-1,Z_t} \otimes S_{p_x,k_x-2,\mathbf Z_x},
\end{equation}
which guarantees
\begin{equation}
\label{eq:inclusion:discrete}
(\partial_t - \kappa \Delta) \widehat Y_h \subseteq \widehat U_h.
\end{equation}
The corresponding function space $U_h \subset U$ for the physical domain is again defined by the pull-back principle, i.e.,
\[
    U_h := \widehat U_h \circ \mathbf G_T^{-1}.
\]
Again, we choose tensor-product bases for $\widehat U_h$ and $U_h$.
It is worth noting that this choice does \emph{not} guarantee
$(\partial_t - \kappa \Delta) Y_h \subseteq  U_h$ for general geometries.
We discretize Problem~\ref{prob:variational} with the Galerkin principle, which leads to the linear system
\[
\mathcal{A}_h \, \underline{\mathbf{x}}_h = \underline{\mathbf{b}}_h,
\]
where
\begin{equation}
\label{eq:3by3h}
\mathcal{A}_h = \begin{pmatrix}
  M_{q_t,h} & 0 & L_h^\top\\
  0& \alpha M_h & M_h&\\
  L_h& M_h & 0
\end{pmatrix},\quad
\underline{\mathbf{x}}_h=
\begin{pmatrix}
\underline y_h\\ \underline u_h\\ \underline\lm_h
\end{pmatrix},\quad
\underline{\mathbf{b}}_h=
\begin{pmatrix}
M_{q_t,h} \underline y_{d,h}\\ 0\\ M_h \underline f_h
\end{pmatrix},
\end{equation}
$M_{q_t,h}$, $M_h$ and $L_h$ are Galerkin mass and stiffness matrices representing the operators $M_{q_t}$, $M$ and $L$, respectively. Here and in what follows, $\underline y_h$, $\underline u_h$ and $\underline\lm_h$ are the representations of the corresponding functions with respect to the chosen bases.
The vectors
$M_{q_t,h} \underline y_{d,h}$ and $M_h \underline f_h$
are obtained by evaluating the corresponding bilinear forms
$(y_d,\cdot)_{L^2(q_t)}$ and $(f,\cdot)_{L^2(Q_T)}$, respectively, for all basis functions.
Analogous to the continuous case, we choose
\begin{equation}
\label{eq:Bh}
\mathcal{B}_h = \begin{pmatrix}
  P_h  & 0 & 0\\
  0& \alpha M_h & 0\\
  0& 0 & \tfrac1\alpha M_h
\end{pmatrix}
\quad\mbox{with}\quad
P_h := M_{q_t,h} + \alpha \kappa^2 B_h + \alpha C_h
\end{equation}
as preconditioner, where $B_h$ and $C_h$ are the Galerkin matrices representing the operators $B$ and $C$, respectively.

In order to provide an upper bound for $\kappa(\mathcal B_h^{-1} \mathcal A_h)$, we use the same approach as for the continuous case: we show that $\alpha L_h^{\top} M_h^{-1} L_h$ and $\alpha \kappa B_h^2 + \alpha C_h$ are spectrally equivalent. Since $U_h\subset U$, the Cauchy-Schwarz inequality, the definitions of the involved matrices and Lemma~\ref{lemma:YZeq} immediately imply
\begin{equation}\nonumber
\begin{aligned}
    &\underline y_h^\top L_h^\top M_h^{-1} L_h \underline y_h
     = \sup_{\underline \lambda_h \in \mathbb R^{\dim U_h}} \frac{ (\underline \lambda_h^\top L_h \underline y_h)^2 }{ \|\underline \lambda_h\|_{M_h}^2}
    = \sup_{\lambda_h \in U_h} \frac{ \langle L y_h, \lambda_h \rangle^2 }{ \|\lambda_h\|_{L^2(Q_T)}^2}
    \\&\qquad\le \sup_{\lambda \in U} \frac{ \langle L y_h, \lambda \rangle^2 }{ \|\lambda\|_{L^2(Q_T)}^2}
    \le \| \partial_t y_h - \kappa \Delta y_h\|_{L^2(Q_T)}^2
    \le 2 \underline y_h^\top
        (   C_h + \kappa^2 B_h  ) \underline y_h,
\end{aligned}
\end{equation}
which shows
\begin{equation}\label{eq:discr:cont}
    L_h^\top M_h^{-1} L_h\le2(   C_h + \kappa^2 B_h  ).
\end{equation}
Next, we need to discuss the other direction. Analogous to the continuous case, we have to show inf-sup stability. Before we consider the general case, we consider a simple special case.
\begin{remark}
If $\Omega=(a_1,b_1)\times\cdots\times (a_d,b_d)$ is a box domain and $\mathbf G_T$ is the corresponding affine-linear mapping $\mathbf G_T(t,\widehat x_1,\cdots,\widehat x_d)=(t, a_1+\widehat x_1(b_1-a_1),\cdots,a_d+\widehat x_d(b_d-a_d))$, the inclusion~\eqref{eq:inclusion:discrete} obviously carries over to the physical domain, i.e., we have
$
(\partial_t - \kappa \Delta) Y_h \subseteq U_{h}.
$
Analogous to the continuous case, we have for the choice $\lambda_h:=Ly_h$ using Lemma~\ref{lemma:YZeq} that
\begin{equation}\nonumber
\begin{aligned}
    &\underline y_h^\top L_h^\top M_h^{-1} L_h \underline y_h
    = \sup_{\lambda_h \in U_h} \frac{ \langle L_h y_h, \lambda_h \rangle^2 }{ \|\lambda_h\|_{L^2(Q_T)}^2}
     = \sup_{\lambda_h \in U_h} \frac{ \langle L y_h, \lambda_h \rangle^2 }{ \|\lambda_h\|_{L^2(Q_T)}^2}
   \ge \|L y_h\|_{L^2(Q_T)}^2
     \\&\qquad= \| \partial_t y_h - \kappa \Delta y_h\|_{L^2(Q_T)}^2 \ge \underline y_h^\top
        (   C_h + \kappa^2 B_h  ) \underline y_h,
\end{aligned}
\end{equation}
which shows together with~\eqref{eq:discr:cont} the spectral equivalence of $L_h^\top M_h^{-1} L_h$ and $\alpha \kappa^2 B_h + \alpha C_h$.
Analogous to Lemma~\ref{lem:3}, we obtain
$\kappa( \mathcal{B}_h^{-1} \mathcal{A}_h ) \le 2 \frac{\cos(\pi/7)}{\cos(\pi/14)} \approx 8.10$.
\end{remark}

In the remainder of this section, we show the following theorem, which guarantees spectral equivalence if the grid is fine enough.
\begin{theorem}\label{thrm:discrinfsup}
    There is a constant $c>0$ only depending on $\mathbf G$, $p_x$, $k_x$ and $h_{\mathbf Z_x}/h_{\mathbf Z_x,\min}$ such that
    \begin{align*}
            \sup_{\lambda_h \in U_h} \frac{ \langle L_h y_h , \lambda_h\rangle^2}{ \|\lambda_h\|_{L^2(Q_T)}^2}
            \ge (1-ch_x^2)  \|\partial_t y_h - \kappa\Delta y_h\|_{L^2(Q_T)}^2
            \quad\foralls y_h\in Y_h,
    \end{align*}
    where $h_x:= h_{\mathbf Z_x}$ is the chosen grid size in space.
\end{theorem}
In order to prove this result, we use the following theorem.
\begin{theorem}\label{thrm:distortion:error}
    Let $V_h:=S_{p,k,\mathbf Z}(\widehat \Omega)$ for some $p>k\ge 0$ and some $\mathbf Z$.    Then, there is a constant $c>0$ only depending on $p$, $k$ and $h_{\mathbf Z}/h_{\mathbf Z,\min}$ such that
    \begin{align*}
            \inf_{v_h \in V_h}
            \|\rho w_h - v_h\|_{L^2(\widehat \Omega)}
            \le c h_{\mathbf Z}\; | \rho |_{W^1_\infty(\widehat \Omega)} \|w_h\|_{L^2(\widehat \Omega)}
   \;
   \foralls w_h\in V_h, \; \rho \in W^1_\infty(\widehat \Omega).
    \end{align*}
\end{theorem}
A proof of this theorem is given in the Appendix.  The following lemma follows from Theorem~\ref{thrm:distortion:error} using standard substitution and chain rules.
\begin{lemma}\label{lem:for:laplace}
        There is a constant $c>0$ only depending on $\mathbf G$, $p$, $k$ and $h_{\mathbf Z}/h_{\mathbf Z,\min}$ such that
        \[
            \inf_{v_h\in U_h}
                    \|-  \Delta y_h-v_h\|_{L^2(Q_T)} \le c h_{\mathbf Z} \|-  \Delta y_h\|_{L^2(Q_T)}\qquad\foralls y_h\in Y_h.
        \]
\end{lemma}
\begin{proof}
        Within this proof, $c$ is a generic constant, only depending on $p$, $k$ and $h_{\mathbf Z}/h_{\mathbf Z,\min}$ that can have different values in different instances.
        First, observe that one can show with standard chain and substitution rules and~\eqref{eq:ass:geo} that
        \[
        \|-  \Delta y_h-v_h\|_{L^2(Q_T)}
        \le c
        \| - \tfrac{1}{|\det \nabla \mathbf G|}  \nabla \cdot (A\nabla \widehat y_h) -\widehat v_h \|_{L^2(\widehat Q_T)},
        \]
        where $A :=\tfrac1{|\det \nabla \mathbf G|} \Cf^\top (\nabla \mathbf G) \Cf( \nabla \mathbf G)$ and $\Cf(\mathbf J) = (\det \mathbf J) \mathbf J^{-\top}$ is the cofactor matrix. Using the triangle inequality, we have further
        \begin{align*}
                &\inf_{v_h \in U_h} \|-  \Delta y_h-v_h\|_{L^2(Q_T)}
                \le c
                \inf_{\widehat v_h\in \widehat U_h}
                \| - \tfrac{1}{|\det \nabla \mathbf G|} \nabla \cdot (A\nabla \widehat y_h) -\widehat v_h \|_{L^2(\widehat Q_T)}
                \\&
                \le
                c\sum_{j=1}^d \inf_{\widehat v_h\in \widehat U_h}  \|\beta_j \tfrac{\partial \widehat y_h}{\partial x_j}-\widehat v_h\|_{L^2(\widehat Q_T)}
                +c\sum_{i,j=1}^d \inf_{\widehat v_h\in \widehat U_h}  \|\alpha_{i,j} \tfrac{\partial^2 \widehat y_h}{\partial x_i\partial x_j}-\widehat v_h\|_{L^2(\widehat Q_T)},
        \end{align*}
        where $\alpha_{i,j}= \tfrac{1}{|\det \nabla \mathbf G|}A $ and $\beta_{i}=\tfrac{1}{|\det \nabla \mathbf G|}\nabla \cdot A$. Since the determinant is uniformly bounded from below (see~\eqref{eq:ass:geo}), we have $\tfrac1{|\det \nabla \mathbf G |}\in W^2_\infty(\widehat \Omega)$. By combining these arguments and since none of these terms depends on the time $t$, we have $A\in W^2_\infty(\widehat Q_T)$. Thus, $\alpha_{i,j}\in W^2_\infty(\widehat Q_T)$ and $\beta_i\in W^1_\infty(\widehat Q_T)$. Since $\alpha_{i,j}$ and $\beta_j$ are independent of $t$, we can apply Theorem~\ref{thrm:distortion:error} for $Q_T$ and obtain
         \begin{align*}
                & \inf_{v_h \in U_h} \|-  \Delta y_h-v_h\|_{L^2(Q_T)}
                \\&\quad
                    \le c h_{\mathbf Z}  \sum_{j=1}^d |\beta_j|_{W^1_\infty(\widehat Q_T)}  |  \widehat y_h|_{H^1(\widehat Q_T)}
                    + c h_{\mathbf Z}\sum_{i,j=1}^d |\alpha_{i,j}|_{W^1_\infty(\widehat Q_T)} | \widehat y_h|_{H^2(\widehat Q_T)}
            \\&\quad \le c  h_{\mathbf Z} \| \widehat y_h\|_{H^2(\widehat Q_T)}.
        \end{align*}
        Using substitution and chain rule and the Poincar\'e inequality, we have further
          \begin{align*}
                &\inf_{v_h \in U_h} \|-  \Delta y_h-v_h\|_{L^2(Q_T)}
                \le c  h \| y_h\|_{H^2(Q_T)}
                \le c  h | y_h|_{H^2(Q_T)}.
        \end{align*}
        Since $y_h\in H^2(\Omega)\cap H^1_0(\Omega)$, we have using integration by parts that $| y_h|_{H^2( \Omega)}=\|-\Delta y_h\|_{L^2( \Omega)}$, which finishes the proof.
\end{proof}

Since $\mathbf G_T$ is defined as a simple lifting of $\mathbf G$, we immediately obtain the
following result.
\begin{lemma}\label{lem:for:time}
        We have $\inf_{v_h\in U_h} \|\partial_t y_h-v_h\|_{L^2(Q_T)} = 0$
        for all $y_h \in Y_h$.
\end{lemma}
\begin{proof}
        Let $y_h \in Y_h$ arbitrary but fixed with $\widehat y_h:=y_h \circ \mathbf G_T$.
        Using the substitution rule \eqref{eq:ass:geo} and the fact that $\mathbf G_T$ is constant in time, we immediately have
        \[
            \inf_{v_h\in U_h}
                    \|\partial_t y_h-v_h\|_{L^2(Q_T)}
                    \le c_{\mathbf G}
            \inf_{\widehat v_h\in \widehat U_h}
                    \| \partial_{ t} \widehat y_h - \widehat v_h\|_{L^2(\widehat Q_T)},
        \]
        which vanishes for all $\widehat y_h \in \widehat Y_h$ since $\partial_{t} \widehat Y_h \subseteq \widehat U_h$.
\end{proof}

Using these results, we can prove Theorem~\ref{thrm:discrinfsup}.

\begin{proofof}{of Theorem~\ref{thrm:discrinfsup}}
First, we observe that an inf-sup statement can be rewritten in terms of an approximation error statement. Let $y_h$ be arbitrarily but fixed and $w_h:=L_h y_h = \partial_t y_h -\kappa\Delta y_h$. Let $\Pi_h$ be the $L^2(Q_T)$-orthogonal projection into $U_h$. By choosing $\lambda_h:=\Pi_h w_h = \Pi_h  (\partial_t y_h-\Delta y_h)$, we obtain
\begin{align*}
        & \sup_{\lambda_h \in U_h} \frac{ \langle L_h y_h , \lambda_h\rangle^2}{ \|\lambda_h\|_{L^2(Q_T)}^2}
                \ge \frac{ (\Pi_h w_h , \Pi_h w_h)_{L^2(Q_T)}^2}{ \|\Pi_h w_h\|_{L^2(Q_T)}^2}
                = \|\Pi_h w_h\|_{L^2(Q_T)}^2
        \\& \quad = \|w_h\|_{L^2(Q_T)}^2 - \|w_h - \Pi_h w_h\|_{L^2(Q_T)}^2 = \|w_h\|_{L^2(Q_T)}^2 - \inf_{v_h\in U_h} \| w_h-v_h\|_{L^2(Q_T)}^2
        \\& \quad \ge \|w_h\|_{L^2(Q_T)}^2 - \left( \kappa
            \inf_{v_h\in U_h} \| -\Delta y_h -v_h\|_{L^2(Q_T)}
            + \inf_{v_h\in U_h} \| \partial_t y_h-v_h\|_{L^2(Q_T)} \right)^2.
\end{align*}
Using Lemmas~\ref{lem:for:laplace}, \ref{lem:for:time} and \ref{lemma:YZeq} and $w_h = \partial_t y_h-\Delta y_h$, we immediately obtain the desired result.
\end{proofof}

Using Theorem~\ref{thrm:discrinfsup}, Lemma~\ref{lemma:YZeq} and the definitions of the involved matrices, we obtain
\begin{equation}\label{eq:discr:infsup}
\begin{aligned}
    &\underline y_h^\top L_h^\top M_h^{-1} L_h \underline y_h
     = \sup_{\underline \lambda_h \in \mathbb R^{\dim U_h}} \frac{ \underline \lambda_h^\top L_h \underline y_h }{ \|\underline \lambda_h\|_{M_h}^2}
    = \sup_{\lambda_h \in U_h} \frac{ \langle L y_h, \lambda_h \rangle }{ \|\lambda_h\|_{L^2(Q_T)}^2}
    \\&\quad\ge (1-ch_x^2) \| \partial_t y_h - \kappa \Delta y_h\|_{L^2(Q_T)}^2
    \\&\quad\ge (1-ch_x^2) (\| \partial_t y_h\|_{L^2(Q_T)}^2 + \|\kappa \Delta y_h\|_{L^2(Q_T)}^2)
    = (1-ch_x^2) \underline y_h^\top
        (   C_h + \kappa^2 B_h  ) \underline y_h.
\end{aligned}
\end{equation}
The combination of~\eqref{eq:discr:cont} and~\eqref{eq:discr:infsup} yields
\[
    (1-ch_x^2)(   C_h + \kappa^2 B_h  )
    \le L_h^\top M_h^{-1} L_h
    \le 2(   C_h + \kappa^2 B_h  ),
\]
which shows equivalence of these matrices if the grid is fine enough. Under this assumption, we obtain -- as in the continuous case -- a condition number bound for the discretized problem:
\[
    \kappa(\mathcal B_h^{-1} \mathcal A_h)
    \le \frac{2 \cos(\pi/7)}{(1-ch_x^2)\sin(\pi/14)} \approx \frac{8.10}{1-ch_x^2},
\]
so, for $h_x\to 0$, the upper bound on the condition number for the discretized problem converges to the upper bound for the condition number for the continuous problem.

\begin{remark}[2--by--2 formulation]
    The same arguments can be applied to the 2--by--2 formulation as well. In this case, the upper bound for the condition number of the preconditioned system is $\frac{2 \cos(\pi/5)}{(1-ch_x^2)\sin(\pi/10)} \approx \frac{5.24}{1-ch_x^2}$.
\end{remark}

\section{Fast diagonalization in time}
\label{sec:fd}

In order to solve the linear system $\mathcal A_h \, \underline{\mathbf x}_h = \underline{\mathbf b}_h$ efficiently using the preconditioner $\mathcal B_h$ as defined in~\eqref{eq:Bh}, we need to be able to realize the application of the preconditioner efficiently, i.e., the application of $P_h^{-1}$ and $M_h^{-1}$ to any vector.
Certainly, these matrices are sparse, symmetric and positive definite, so any standard method could be applied. However, it is worth noting that that linear system can be quite large since it involves the whole space-time cylinder. We propose to adopt the fast diagonalization method, see~\cite{sangalli2016isogeometric}, in time to decouple problem on the whole space-time cylinder. Since we assume the fast diagonalization method to be well-known, we sketch the ideas only briefly.

The matrix $M_{q_t,h}$ is obtained by evaluating $(\cdot,\cdot)_{L^2(q_t)}$ for the basis functions, specifically $[M_{q_t,h}]_{\tilde n,n} = (\psi_n,\psi_{\tilde n})_{L^2(q_t)}$. Recall that the basis functions $\psi_n$ have tensor-product structure~\eqref{eq:tps}, i.e., $\psi_{iN_x+j}(t,x)=\varphi_i(t)\;\phi_j(x)$.
Assuming that the observation that is distributed in $\Omega$, i.e.,
\[
    q_t = \mathcal T_{\mathcal O} \times \Omega \quad \mbox{for some}\quad \mathcal T_{\mathcal O} \subseteq [0,T],
\]
we have
\[
    [M_{q_t,h}]_{\tilde i N_x+\tilde j,iN_x+j}
    =(\psi_{iN_x+j},\psi_{\tilde iN_x+\tilde j})_{L^2(q_t)}
    =
    \underbrace{(\varphi_{i},\varphi_{\tilde i})_{L^2(T_{\mathcal O})}}_{\displaystyle [M_{q_t,t}]_{\tilde i,i:=}}
    \underbrace{(\phi_{j},\psi_{\tilde j})_{L^2(\Omega)}}_{\displaystyle [M_x]_{\tilde j,j}:=}
\]
with mass matrices $M_x$ and $M_{q_t,t}$ in space and time, respectively. Using the notation of a Kronecker product, we just write this equivalently as
\[
        M_{q_t,h} = M_{q_t,t} \otimes M_x.
\]
With completely analogous arguments, we obtain
\[
    B_h = M_t \otimes B_x
    \quad \mbox{and}\quad
    C_h = K_t \otimes M_x,
\]
where $B_x$ is a fourth order stiffness matrix in space, $K_t$ is a second order stiffness matrix in time and $M_x$ and $M_t$ are mass matrices in space and time.
Using these decompositions, we obtain
\begin{equation}
    \label{eq:PyKron}
    P_h=M_{q_t,h} + \alpha \kappa^2 B_h + \alpha C_h
        = (M_{q_t,t}+\alpha K_t) \otimes M_x + \alpha \kappa^2 M_t\otimes B_x.
\end{equation}
In order to apply $P_h^{-1}$ to a vector efficiently, we use the fast diagonalization approach.
Since all of these matrices
are symmetric positive definite and thus,
the pencil $(M_{q_t,t} + \alpha K_t;\, M_t)$ has a generalized eigendecomposition; so there is a matrix $U_t$ and a diagonal matrix $D_t=\mbox{diag}(d_1,\cdots,d_{N_t})$ such that
\begin{equation}
\label{eq:GEDtime}
M_{t} = U^{-\top}_{t} U^{-1}_{t}
\quad \text{and} \quad
M_{q_t,t}+\alpha K_{t} = U^{-\top}_{t} D_{t}U^{-1}_{t}.
\end{equation}
By substituting this into \eqref{eq:PyKron}, we obtain
\begin{align*}
    P_h &= U_t^{-\top} D_t U_t^{-1} \otimes M_x + \alpha \kappa^2 U_t^{-\top} U_t^{-1} \otimes B_x\\
    &= (U_t^{-\top}\otimes \I)(D_t \otimes  M_x + \alpha \kappa^2 \I \otimes B_x)(U_t^{-1}\otimes \I),
\end{align*}
where $\I$ is the identity matrix.
The application of $P^{-1}_h$ to a vector $\underline{r}_h$, that is; solving
$
P_h \underline{s}_h = \underline{r}_h
$
is done with following the steps:
\begin{algorithm}[H]
\caption{Fast diagonalization in time}\label{alg:DFtime}
\begin{algorithmic}[1]
    \State Compute the generalized eigendecomposition \eqref{eq:GEDtime}
    \State Set $\underline{\tilde{r}}_h = (U_t\otimes \I)^\top \underline{r}_h$
    \State Solve $(D_t\otimes M_x + \alpha \kappa^2 \I\otimes B_x)\underline{\tilde s}_h = \underline{\tilde{r}}_h$
    \State Set $\underline{s}_h = (U_t\otimes \I) \underline{\tilde{s}}_h$
\end{algorithmic}
\end{algorithm}

Note that Step~1 of the algorithm has to be done only once, even if the preconditioner has to be applied several times. Its computational costs are independently of the problem size in the space direction. The application of Steps~2 and 4 is done independently for each of the degrees of freedom in space; thus, it can be easily realized in parallel. Finally, the application of Step~3 is done independently for each point in the frequency domain. So, we have to solve $N_t$ linear systems
\[
    (d_j M_x + \alpha \kappa^2 B_x)  \underline {\tilde s}_h^{(j)} = \underline {\tilde r}_h^{(j)},
\]
where $d_1,\cdots,d_{N_t}>0$ are the diagonal entries of $D_t$ and $\underline {\tilde r}_h^{(j)}$ and $\underline {\tilde s}_h^{(j)}$ are the corresponding blocks of $\underline {\tilde r}_h$ and $\underline {\tilde s}_h$, respectively.
These linear systems correspond to the following variational problems in space: Find $w_h \in W_h := \{ w_h: w_h\circ \mathbf G \in S_{p_x,t_x,\mathbf Z_x} \mbox{ and } w_h|_{\partial\Omega}=0\}$ such that
\begin{equation}
\label{eq:Biharmonic}
   d_j (w_h,v_h)_{L^2(\Omega)} + \alpha \kappa^2 (\Delta w_h, \Delta v_h)_{L^2(\Omega)}
   = (g_h,v_h)_{L^2(\Omega)} \quad\forall\; v_h\in W_h.
\end{equation}
We are free to choose any solution strategy for solving that problem.
\begin{remark}[Fast diagonalization in space]
    One possibility for solving the problems~\eqref{eq:Biharmonic} is to apply fast diagonalization also in space. While the proposed fast diagonalization approach in time is exact, a fast diagonalization approach in space is inexact for two reasons.
    First, the involved matrices have a Kronecker-product structure only for affine-linear geometry functions $\mathbf G(x_1,\ldots, x_d)$, so the fast diagonalization preconditioner would need to approximate the geometry function in any other case. Numerical experiments have shown that the effects of such an approximation are more severe for biharmonic problems, compared to the Poisson problem.
    Second, an expansion of the bilinear form $(\Delta u_h, \Delta v_h)_{L^2(\Omega)}$ also yields mixed terms, like $(\partial_{xx} u_h, \partial_{yy} v_h)_{L^2(\Omega)}$ for $d=2$. It is possible to show that discarding these terms yields a robust approximation, an approximation nonetheless.
\end{remark}
\begin{remark}[Multigrid solvers]
    An alternative for solving~\eqref{eq:Biharmonic} are multigrid solvers. In the next section, we present results for multigrid solvers with standard Gauss-Seidel smoothers. It is worth noting that these multigrid solvers are not robust with respect to the spline degree $p$. Robust multigrid solvers for the biharmonic problem were proposed in \cite{SognTakacsBiharmonic2}; however, this approach suffers from similar issues as fast diagonalization if the geometry function is not affine-linear.
\end{remark}
\begin{remark}[Computational costs]
Finally, we want to discuss the computational costs of Algorithm~\ref{alg:DFtime}. Let $N_t$ be the number of degrees of freedom in time direction and $N_x$ be the number of degrees of freedom in the space directions. Assuming standard algorithms and keeping in mind that $U_t$ is a dense matrix, Step~1 of the algorithm (computation of the eigendecomposition) costs $\mathcal O(N_t^3)$ flops, the Steps~2 and 4 of the algorithm cost $\mathcal O(N_t^2 N_x)$ flops. Step~3 requires the solution of $N_t$ elliptic problems. With multigrid solvers, each such problem can be solved with $\mathcal O(N_x)$ flops. By applying this approach, one would obtain an overall complexity of
\[
        \mathcal O(N_t^2N_x+N_t^3).
\]
Since one would expect $N_t\le N_x$, the first term dominates. Keeping in mind that the total number of degrees of freedom in $N=N_tN_x$, the overall complexity would be $\mathcal O(N_tN)$ in this case.
\end{remark}
\begin{remark}[Other kinds of observation domains]
The approach proposed in this section is also possible for the case discussed in Remark~\ref{rem:ti}.
If the observation is distributed in time, but limited in space (this includes the case of boundary observation discussed in Remark~\ref{rem:bo}), we have $M_{q_t,h} = M_t \otimes M_{q_t,x}$, which can be represented as follows
\[
    P_h = \alpha K_t \otimes M_x + M_t \otimes ( M_{q_t,x} +\alpha \kappa^2 B_x).
\]
This allows for an application of the fast diagonalization algorithm completely analogous to the case discussed in this section.
\end{remark}

\section{Numerical results}
\label{sec:numeric}

In the following, we give results of numerical experiments that illustrate the presented theory.
All presented results have been implemented using the G+Smo library \cite{gismoweb}.
As computational domain $\Omega$, we choose the approximation to a quarter annulus domain depicted in Figure~\ref{fig:annulus}, which is parameterized with a B-spline of degree $2$ without inner knots. The space-time cylinder is $(0,1]\times \Omega$. As state equation, we consider the heat equation with homogeneous Dirichlet boundary conditions:
\begin{align*}
    \partial_t y - \kappa \Delta y & = u \quad\mbox{in}\quad (0,1]\times \Omega, \\
    y&=0 \quad \mbox{on} \quad (0,1]\times \partial \Omega,\\
    y&=y_0 \quad \mbox{on} \quad \{0\} \times \Omega,
\end{align*}
where $y_0$ is the characteristic function for the union of three circles with radius $0.2$ around the points $p_1$, $p_2$ and $p_3$ with
\[
            p_i= (\tfrac32 \cos (i\pi/8),
                  \tfrac32 \sin (i\pi/8)),
                  \qquad i=1,2,3,
\]
i.e., it is defined by
\[
    y_0(x) :=
    \begin{cases}
        1 & \mbox{ if } |x-p_i| < 0.2 \mbox{ for one } i\in\{1,2,3\},\\
        0 & \mbox{ otherwise,}
    \end{cases}
\]
see Figure~\ref{fig:ic} for the visualization of the projection of $y_0$ into a fine spline space (see below for details on the spline space).
\begin{figure}[htp]
\centering
\begin{minipage}{.45\textwidth}
  \centering
    \begin{tikzpicture}
        \draw (0, 0) node[inner sep=0] {\includegraphics[height=10em]{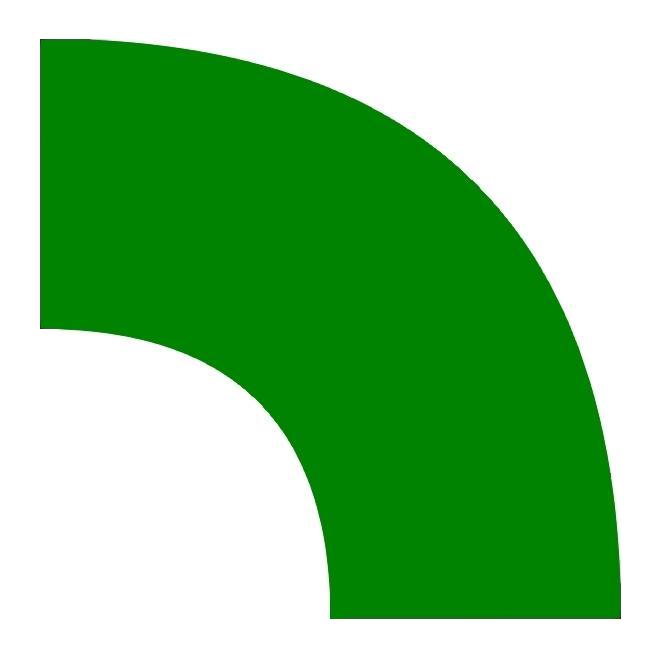}};

        \draw (-1.56, -1.57) -- (1.56, -1.57);
        \draw (-1.56, -1.65) -- (-1.56, -1.57);
        \draw (-1.56, -1.75) node {\tiny 0};
        \draw (0, -1.65) -- (0, -1.57);
        \draw (0, -1.75) node {\tiny 1};
        \draw (1.56, -1.65) -- (1.56, -1.57);
        \draw (1.56, -1.75) node {\tiny 2};

        \draw (-1.56, -1.57) -- (-1.56, 1.55);
        \draw (-1.63, -1.57) -- (-1.56, -1.57);
        \draw (-1.7, -1.57) node {\tiny 0};
        \draw (-1.63,0) -- (-1.56,0);
        \draw (-1.7, 0) node {\tiny 1};
        \draw (-1.63, 1.55) -- (-1.56, 1.55);
        \draw (-1.7, 1.55) node {\tiny 2};
    \end{tikzpicture}
  \captionof{figure}{Quarter annulus domain\\\;\\\; }
  \label{fig:annulus}
\end{minipage}\quad
\begin{minipage}{.45\textwidth}
    \centering
    \begin{tikzpicture}
        \draw (0, 0) node[inner sep=0] {\includegraphics[height=10em]{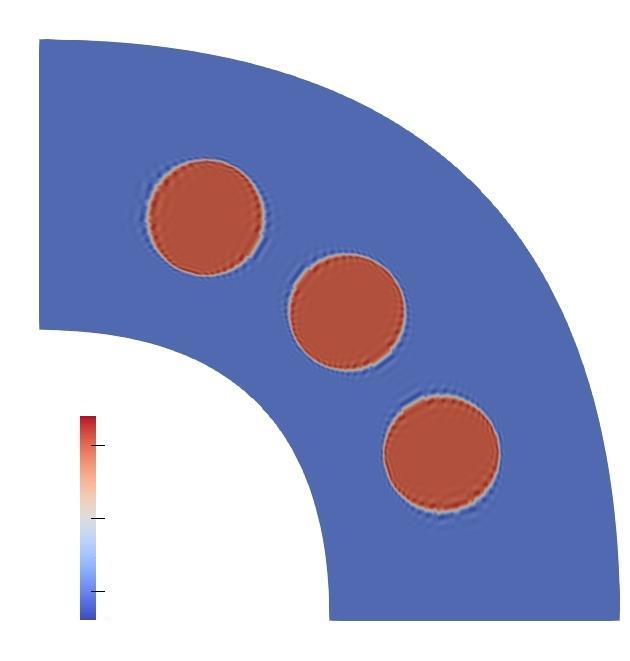}};
        \draw (-.95, -.6) node {\tiny 1\hspace{-.6px}.\hspace{-.6px}0};
        \draw (-.95, -1) node {\tiny 0\hspace{-.6px}.\hspace{-.6px}5};
        \draw (-.95, -1.38) node {\tiny 0\hspace{-.6px}.\hspace{-.6px}0};
    \end{tikzpicture}
    \caption{Initial condition $y_0$ projected into the spline space with $p=2$ and $\ell = 7$.}
    \label{fig:ic}
\end{minipage}
\end{figure}
As desired state $y_d$, we choose the solution of the state equation with the same initial condition and with vanishing source term, i.e., it satisfies
\begin{align*}
    \partial_t y_d - \kappa \Delta y_d & = 0 \quad\mbox{in}\quad (0,1]\times \Omega, \\
    y_d&=0 \quad \mbox{on} \quad (0,1]\times \partial \Omega,\\
    y_d&=y_0 \quad \mbox{on} \quad \{0\} \times \Omega.
\end{align*}
As observation domain, we choose
\begin{equation}
\label{eq:timeIter}
q_t = \left( (0,\tfrac{1}{16})\cup (\tfrac{4}{16},\tfrac{5}{16})\cup (\tfrac{12}{16},\tfrac{13}{16})\cup (\tfrac{15}{16},1) \right) \times \Omega,
\end{equation}
i.e., the control is distributed over the whole computational domain $\Omega$, but restricted to four time intervals.

The variational problem is discretized as follows. As initial grid (grid level $\ell=0$), we choose globally polynomial functions both in space and time, i.e., there are no inner knots. The following grid levels ($\ell=1,2,\ldots$) are obtained by uniform refinement, both in space and time. We choose the spline degree $p\ge 2$ to by uniform in space and time as well. The space for the state variable is chosen to be of maximum smoothness, so we choose
\[
    Y_h
    := \big\{
    y_h \,:\,
    y_h \circ \mathbf G \in
    S_{p,p-1,Z_t} \otimes S_{p,p-1,\mathbf Z_x},\,
    y_h|_{t=0}=0 \mbox{ and } y_h |_{(0,1]\times \partial\Omega}=0
    \big\}.
\]
Following the theory, we choose for the control and the adjoined state the space
\[
    U_h
    := \big\{
    u_h \,:\,
    u_h \circ \mathbf G \in
    S_{p,p-2,Z_t} \otimes S_{p,p-3,\mathbf Z_x}
    \big\}.
\]

Next, we have a look on a possible solution to the problem. Consider the case $p=2$ and $\ell=7$; here the space $Y_h$ has $2\,113\,536$ degrees of freedom ($129$ in time $\times$ $16\,384$ in space) and $U_h$ has $37\,896\,129$ degrees of freedom ($257$ in time $\times$ $147\,456$ in space). For this choice, we obtain the initial condition as depicted in Figure~\ref{fig:ic}. For the choice $\kappa=10^{-2}$ and $\alpha=10^{-3}$, the optimal state looks as depicted in Figure~\ref{fig:solution:sp3r7Limk2a3}.
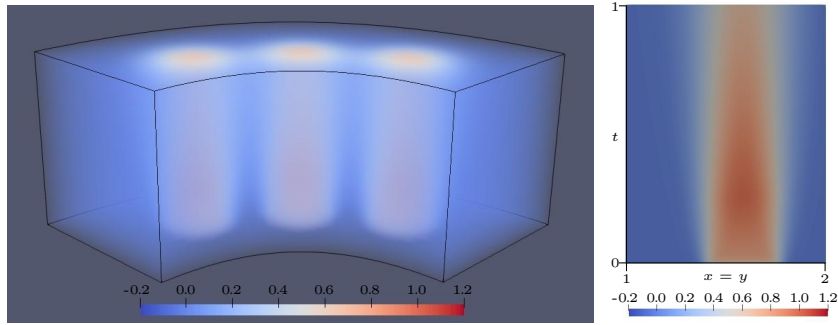
\begin{figure}[htp]
    \centering
    \begin{tikzpicture}
        \draw (0, 0) node[inner sep=0] {\includegraphics[height=12em]{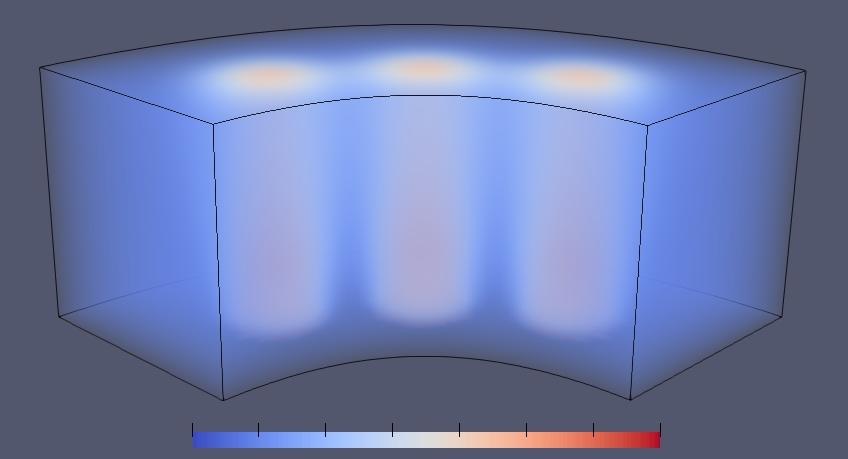}};
        \draw (-2.17, -1.67) node {\tiny -0\hspace{-.6px}.\hspace{-.6px}2};
        \draw (-1.52, -1.67) node {\tiny 0\hspace{-.6px}.\hspace{-.6px}0};
        \draw (-.91, -1.67) node {\tiny 0\hspace{-.6px}.\hspace{-.6px}2};
        \draw (-.29, -1.67) node {\tiny 0\hspace{-.6px}.\hspace{-.6px}4};
        \draw (.32, -1.67) node {\tiny 0\hspace{-.6px}.\hspace{-.6px}6};
        \draw (0.93, -1.67) node {\tiny 0\hspace{-.6px}.\hspace{-.6px}8};
        \draw (1.55, -1.67) node {\tiny 1\hspace{-.6px}.\hspace{-.6px}0};
        \draw (2.16, -1.67) node {\tiny 1\hspace{-.6px}.\hspace{-.6px}2};
    \end{tikzpicture}
    \begin{tikzpicture}
        \draw (0, 0) node[inner sep=0] {\includegraphics[height=12.7em]{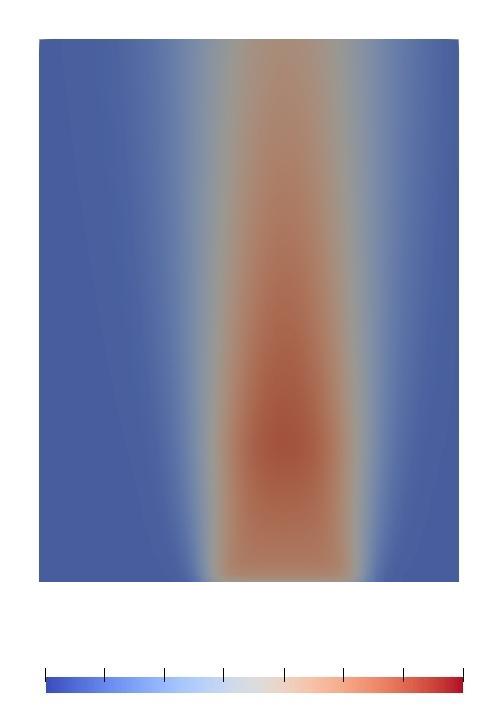}};
        \draw (-1.32, -1.44) -- (1.325, -1.44);
        \draw (-1.32, -1.44) -- (-1.32, 1.98);
        \draw (-1.32, -1.44) -- (-1.32, -1.55);
        \draw (-1.32, -1.65) node {\tiny $1$};
        \draw (0, -1.65) node {\tiny $x=y$};
        \draw (1.32, -1.44) -- (1.32, -1.55);
        \draw (1.32, -1.65) node {\tiny $2$};
        \draw (-1.32, -1.44) -- (-1.42, -1.44);
        \draw (-1.47, -1.44) node {\tiny $0$};
        \draw (-1.47, 0.23) node {\tiny $t$};
        \draw (-1.32, 1.98) -- (-1.42, 1.98);
        \draw (-1.47, 1.97) node {\tiny $1$};
        \draw (-1.34, -1.9) node {\tiny -0\hspace{-.6px}.\hspace{-.6px}2};
        \draw (-.92, -1.9) node {\tiny 0\hspace{-.6px}.\hspace{-.6px}0};
        \draw (-.54, -1.9) node {\tiny 0\hspace{-.6px}.\hspace{-.6px}2};
        \draw (-.16, -1.9) node {\tiny 0\hspace{-.6px}.\hspace{-.6px}4};
        \draw (.22, -1.9) node {\tiny 0\hspace{-.6px}.\hspace{-.6px}6};
        \draw (.60, -1.9) node {\tiny 0\hspace{-.6px}.\hspace{-.6px}8};
        \draw (.98, -1.9) node {\tiny 1\hspace{-.6px}.\hspace{-.6px}0};
        \draw (1.36, -1.9) node {\tiny 1\hspace{-.6px}.\hspace{-.6px}2};
    \end{tikzpicture}
    \caption{Computed state solution in the whole space time cylinder (left) and a cut through the center circle (right).}
    \label{fig:solution:sp3r7Limk2a3}
\end{figure}

Certainly, the results on the performance of the preconditioner are more in the focus of this paper. We set up a minimum residual (minres) solver in order to solve the 3--by--3 formulation of the KKT system. We use the proposed preconditioner, where we apply fast diagonalization in time. The problems in space are solved with a direct solver (using a sparse Cholesky factorization). As stopping criterion, we choose the reduction of the initial residual by a factor of $10^{-6}$. In the first five columns of Table~\ref{table:1}, we show the number of iterations required to reach the desired reduction factor for the fixed discretization ($\ell=6$, $p=2$). As predicted by the theory, the numbers are uniformly bounded for the whole range of considered values of $\alpha$ and $\kappa$. As one would expect from the theory, similar results are obtained for the 2--by--2 formulation, but with smaller iteration counts in general, see the last five columns in Table~\ref{table:1}.

\begin{table}[htb]
\centering
\newcolumntype{R}{>{\raggedleft\let\newline\\\arraybackslash\hspace{0pt}}m{.6cm}}
\begin{tabular}{l|RRRRR|RRRRR}
\hline
 & \multicolumn{5}{c|}{3--by--3}
 & \multicolumn{5}{c}{2--by--2}\\
\hline
 & \multicolumn{5}{c|}{$\kappa$}
 & \multicolumn{5}{c}{$\kappa$}\\
   $ \alpha$
               &\hspace{-1em}$10^{1}$& \hspace{-1em}$10^{0}$& \hspace{-1em}$10^{-1}$& \hspace{-1em}$10^{-2}$& \hspace{-1em}$10^{-3}$
               &\hspace{-1em}$10^{1}$& \hspace{-1em}$10^{0}$& \hspace{-1em}$10^{-1}$& \hspace{-1em}$10^{-2}$& \hspace{-1em}$10^{-3}$ \\ \hline
   $10^{0}$    &  $38$        &  $40$        &  $38$        &  $32$        &  $31$
               &  $25$        &  $27$        &  $27$        &  $23$        &  $21$\\
   $10^{-1}$   &  $33$        &  $38$        &  $37$        &  $37$        &  $34$
               &  $23$        &  $25$        &  $25$        &  $19$        &  $19$\\
   $10^{-2}$   &  $31$        &  $34$        &  $42$        &  $43$        &  $43$
               &  $21$        &  $23$        &  $23$        &  $17$        &  $17$\\
   $10^{-3}$   &  $26$        &  $39$        &  $46$        &  $48$        &  $48$
               &  $19$        &  $21$        &  $19$        &  $15$        &  $15$\\
   $10^{-4}$   &  $27$        &  $42$        &  $53$        &  $55$        &  $48$
               &  $15$        &  $17$        &  $17$        &  $13$        &  $11$\\
   $10^{-5}$   &  $34$        &  $45$        &  $54$        &  $55$        &  $45$
               &  $13$        &  $13$        &  $13$        &  $11$        &  $7$\\
   $10^{-6}$   &  $41$        &  $50$        &  $57$        &  $53$        &  $30$
               &  $17$        &  $13$        &  $17$        &  $11$        &  $27$\\
\hline
\end{tabular}
\caption{\label{table:1} Number of minres iterations; $\ell=6$; $p=2$.}
\end{table}

\begin{table}[htb]
\centering
\newcolumntype{R}{>{\raggedleft\let\newline\\\arraybackslash\hspace{0pt}}m{.5cm}}
\begin{tabular}{l|RRRRR|RRRRR}
\hline
 & \multicolumn{5}{c|}{3--by--3}
 & \multicolumn{5}{c}{2--by--2}\\
\hline
 & \multicolumn{5}{c|}{$\ell$}
 & \multicolumn{5}{c}{$\ell$}\\
   $p$         &  $2$         &  $3$         &  $4$         &  $5$         &  $6$
               &  $2$         &  $3$         &  $4$         &  $5$         &  $6$  \\ \hline
   $2$         &  $33$        &  $40$        &  $47$        &  $50$        &  $48$
               &  $8$         &  $9$         &  $13$        &  $15$        &  $15$  \\
   $3$         &  $37$        &  $44$        &  $49$        &  $50$        &  $49$
               &  $8$         &  $10$        &  $13$        &  $15$        &  $15$\\
   $4$         &  $41$        &  $44$        &  $50$        &  $50$        &  $49$
               &  $10$        &  $11$        &  $13$        &  $15$        &  $17$\\
   $5$         &  $43$        &  $45$        &  $50$        &  $50$        &  $49$
               &  $9$         &  $11$        &  $13$        &  $15$        &  $15$\\
\hline
\end{tabular}
\caption{\label{table:2} Number of minres iterations; $\kappa=10^{-2}$; $\alpha=10^{-3}$.}
\end{table}

\begin{table}[htb]
\centering
\newcolumntype{R}{>{\raggedleft\let\newline\\\arraybackslash\hspace{0pt}}m{.78cm}}
\begin{tabular}{l|RRR|RRR|RRR}
\hline
 & \multicolumn{3}{c|}{partial obs.; $U_h$}
 & \multicolumn{3}{c|}{full obs.; $U_h$}
 & \multicolumn{3}{c}{partial obs.; $\widetilde U_h$\Strut}\\
\hline
 & \multicolumn{3}{c|}{$\kappa$}
 & \multicolumn{3}{c|}{$\kappa$}
 & \multicolumn{3}{c}{$\kappa$}\\
   $ \alpha$
               &  $10^1$      &  $10^{-1}$   &  $10^{-3}$
               &  $10^1$       &  $10^{-1}$   &  $10^{-3}$
               &  $10^1$       &  $10^{-1}$   &  $10^{-3}$\\ \hline
   $10^{ 0}$   &  $1.99$      &  $1.99$      &  $1.95$
               &  $1.99$      &  $1.99$      &  $1.95$
               &  $7.78$      &  \hspace{-1em}$370.46$    &  \hspace{-1em}$350.08$\\
   $10^{-1}$   &  $1.99$      &  $1.98$      &  $1.92$
               &  $1.99$      &  $1.98$      &  $1.92$
               &  $7.78$      &  \hspace{-1em}$355.58$    &  \hspace{-1em}$332.98$\\
   $10^{-2}$   &  $1.98$      &  $1.97$      &  $1.90$
               &  $1.98$      &  $1.97$      &  $1.90$
               &  $7.76$      &  \hspace{-1em}$321.83$    &  \hspace{-1em}$299.80$\\
   $10^{-3}$   &  $1.96$      &  $1.94$      &  $1.81$
               &  $1.96$      &  $1.94$      &  $1.81$
               &  $7.59$      &  \hspace{-1em}$239.94$    &  \hspace{-1em}$224.48$\\
   $10^{-4}$   &  $1.83$      &  $1.81$      &  $1.48$
               &  $1.83$      &  $1.81$      &  $1.48$
               &  $6.88$      &  \hspace{-1em}$104.75$    &  \hspace{-1em}$82.72$\\
   $10^{-5}$   &  $1.52$      &  $1.50$      &  $1.14$
               &  $1.52$      &  $1.50$      &  $1.14$
               &  $5.60$      &  \hspace{-1em}$57.62$     &  \hspace{-1em}$44.61$\\
   $10^{-6}$   &  $1.20$      &  $1.19$      &  $1.02$
               &  $1.20$      &  $1.19$      &  $1.02$
               &  $4.37$      &  \hspace{-1em}$40.50$     &  \hspace{-1em}$35.68$\\
\hline
\end{tabular}
\caption{\label{table:3} Estimated condition number; $\ell=6$; $p=2$}
\end{table}

\begin{table}[htb]
\centering
\newcolumntype{R}{>{\raggedleft\let\newline\\\arraybackslash\hspace{0pt}}m{.85cm}}
\begin{tabular}{l|RRR|RRR|RRR}
\hline
 & \multicolumn{3}{c|}{partial observation; $U_h$}
 & \multicolumn{3}{c|}{full observation; $U_h$}
 & \multicolumn{3}{c}{partial observation; $\widetilde U_h$\Strut}\\
\hline
 & \multicolumn{3}{c|}{$\ell$}
 & \multicolumn{3}{c|}{$\ell$}
 & \multicolumn{3}{c}{$\ell$}\\
   $ p$
               &  $2$         &  $4$         &  $6$
               &  $2$         &  $4$         &  $6$
               &  $2$         &  $4$         &  $6$ \\ \hline
   $2$         &  $1.04$      &  $1.66$      &  $1.90$
               &  $1.04$      &  $1.66$      &  $1.90$
               &  $1.41$      &  $15.82$     &\hspace{-1em}$231.32$\\
   $3$         &  $1.09$      &  $1.74$      &  $1.96$
               &  $1.09$      &  $1.74$      &  $1.96$
               &  $1.27$      &  $8.49$      &\hspace{-1em}$141.61$\\
   $4$         &  $1.17$      &  $1.71$      &  $1.98$
               &  $1.17$      &  $1.71$      &  $1.98$
               &  $1.27$      &  $6.36$      &  $94.50$\\
   $5$         &  $1.27$      &  $1.79$      &  $1.99$
               &  $1.27$      &  $1.79$      &  $1.99$
               &  $1.34$      &  $4.68$      &  $66.78$\\
\hline
\end{tabular}
\caption{\label{table:4} Estimated condition number; $\kappa=10^{-2}$; $\alpha=10^{-3}$}
\end{table}

In Table~\ref{table:2}, we present the convergence behaviour with respect to the grid size $h_x\approx h_t \approx 2^{-\ell}$ and the spline degree $p$. Also here, the iteration counts are uniformly bounded, as predicted by the convergence theory.

The main factor for the convergence behavior is the approximation of the Schur complement
$
    S_h:=M_{q_t,h} + \alpha^{-1} L_h^\top M_h^{-1} L_h
$
by means of the preconditioner $P_h$. The quality of this approximation can be measured by the relative condition number $\kappa(P_h^{-1} S_h)$, which we have estimated utilizing a conjugate gradient solver. We present the corresponding results in Tables~\ref{table:3} and \ref{table:4}. In the first three columns, we present the results for the proposed method and the model problem as stated above. We can observe that the condition numbers are well bounded; specifically, the condition number is smaller than $2$ in all cases.

For comparison, we present the condition numbers obtained for full observation in the following three columns. One can observe that the first two digits for condition numbers obtained for limited and full observation are equal. This means that the preconditioner works equally well for the case of full observation and the case of partial observation.

Finally, we investigate what happens if we choose a smaller space for the control and the adjoined state. Specifically, we choose
\begin{equation}\nonumber
    \widetilde U_h
    := \big\{
    u_h \,:\,
    u_h \circ \mathbf G \in
    S_{p,p-1,Z_t} \otimes S_{p,p-1,\mathbf Z_x}
    \big\}
\end{equation}
where $\widetilde U_h$ is basically the same space as $Y_h$ (just without the initial and boundary conditions).
We present the corresponding results in the last three columns of Tables~\ref{table:3} and \ref{table:4}. We observe that this choice leads to much larger condition numbers, specifically if the grid gets refined. Correspondingly, this approach also leads to much larger iteration counts (which we do not present for brevity).

\begin{table}[htb]
\centering
\newcolumntype{R}{>{\raggedleft\let\newline\\\arraybackslash\hspace{0pt}}m{.85cm}}
\begin{tabular}{l|RRRR|RRRR}
\hline
 & \multicolumn{4}{c|}{Cholesky}
 & \multicolumn{4}{c}{multigrid}\\
\hline
 & \multicolumn{4}{c|}{$\alpha$}
 & \multicolumn{4}{c}{$\alpha$}\\
    $\ell$  &$10^0$&$10^{-2}$&$10^{-4}$&$10^{-6}$& $10^0$&$10^{-2}$&$10^{-4}$&$10^{-6}$\\
    \hline
    $2$     &  $28$& $36$    & $29$    & $24$    & $40$  & $47$    & $50$    & $46$ \\
    $3$     &  $29$& $41$    & $37$    & $29$    & $46$  & $48$    & $62$    & $63$ \\
    $4$     &  $30$& $44$    & $48$    & $28$    & $46$  & $52$    & $59$    & $64$ \\
    $5$     &  $30$& $43$    & $52$    & $39$    & $65$  & $57$    & $65$    & $75$ \\
\hline
\end{tabular}
\caption{\label{table:5} Iteration counts for the 3--by--3 formulation; $\kappa=10^{-2}$, $p=2$}
\end{table}

Finally, we want to illustrate what happens if the biharmonic problems~\eqref{eq:Biharmonic} in space are solved by means of iterative solvers. Here, we apply one geometric multigrid W-cycle in order to approximate the solution to~\eqref{eq:Biharmonic}. As a smoother, we have chosen the Gauss-Seidel relaxation, which is known to degrade if the spline degree is increased, cf.~\cite{SognTakacsBiharmonic2}. So, we restrict ourselves to $p=2$, and choose $\kappa=10^{-2}$. We perform $2$ pre- and $2$ post-smoothing steps. In the right four columns, we give the iteration counts for the overall minres solver used to solve the 3--by--3 formulation of the optimality system. As comparization, we present the results obtained with a direct solver (based on a Cholesky factorization) in the first four columns. We observe that replacing the direct solver by the application of a W-cycle increases the iteration counts slightly, like by a factor of two. For large problems, particularly in three spacial dimensions, this might be easily outweighed by the fact that a multigrid W-cycle is usually much cheaper than a full Cholesky factorization, both in terms of computational costs and memory consumption. In order to reduce communication effort in a parallel-in-time realization of the whole preconditioner, one could consider to apply of several multigrid W-cycles, in order to realize the preconditioner.

\section{Conclusions}
\label{sec:conclusions}

In this paper, we have presented a robust (with respect to the model parameters $\alpha,\kappa$, the grid sizes $h_x,h_t$ and the spline degree $p$) preconditioner for optimal control problems of tracking type with parabolic state equation. We have given the analysis both for the continuous and the discrete case. The theory states that it is sufficient to choose $U_h$ such that $(\partial_t-\kappa\Delta) \widehat Y_h\subseteq  \widehat U_h$ in order to get a robust preconditioner. This means that $U_h$ has to be much larger than $Y_h$. The numerical experiments suggest that this is also necessary; specifically, it its not sufficient if $U_h$ and $Y_h$ are chosen to be (up to the boundary conditions) the same.

In order to computationally realize the proposed preconditioner, one has to solve an elliptic problem in space and time which is of second order in time and of fourth order in space. Using fast diagonalization, this can be transformed into a sequence of fourth order elliptic problems in space only. Each of them can be solved in parallel. While, usually, fast diagonalization is only constructed to approximate the corresponding operator, we use fast diagonalization in time only and use the fact that one has a perfect tensor-product for many practical applications. So, the fast diagonalization approach does not only approximate the full operator, but it exactly represents it.
For solving the problems in space, apart from direct solvers, any iterative solver can be used; we have illustrated this by using a multigrid solver.

\section*{Appendix}

In this appendix we prove Theorem~\ref{thrm:distortion:error}. First, we recall that B-splines can be bounded from below with the Euclidean norm of the coefficient vector. Despite that the involved constant degrades exponentially with the spline degree, the dependence with respect to the grid size $h$ is as expected.
\begin{theorem}\label{thrm:scherer}
        Let $V_h:=S_{p,k,\mathbf Z}(\widehat \Omega)$ for some $p>k\ge 0$ and some $\mathbf Z$. Let $(\phi_\ell)_{\ell=1}^L$ be the corresponding B-spline basis (as obtained by the Cox-de Boor formula). Then, there is a constant $c>0$ only depending on $p$, $k$ and $h_{\mathbf Z}/h_{\mathbf Z,\min}$ such that
        $
                \sum_{\ell=1}^L v_\ell^2
                \le c h_{\mathbf Z}^{-d}
                \left\| \sum_{\ell=1}^L v_\ell \phi_\ell \right\|_{L^2(\widehat\Omega)}^2
        $.
\end{theorem}
\begin{proof}
        This result follows directly from \cite[Theorem~1]{SchererSchadrin1999},
        where it was shown that the condition number of the B-spline basis does not grow faster than $p 2^p$:
        \[
             \sup_{v_1,\ldots,v_L}
                \frac{ \sum_{\ell=1}^L v_\ell^2 }
                        { \left\| \sum_{\ell=1}^L v_\ell \phi_\ell \right\|_{L^2(\widehat\Omega)}^2 }
            \;
             \sup_{w_1,\ldots,w_L}
                \frac{ \left\| \sum_{\ell=1}^L w_\ell \phi_\ell \right\|_{L^2(\widehat\Omega)}^2 }
                        { \sum_{\ell=1}^L w_\ell^2 }
            \le c = \mathcal O(p^d 2^{dp}).
        \]
        A result of the desired kind is obtained for the choice $w_1=\cdots=w_L=1$. The extension to tensor-product splines is straight-forward. 
\end{proof}

Now, we can show the desired result.

\begin{proofof}{of Theorem~\ref{thrm:distortion:error}}
    Let $(\phi_\ell)_{\ell=1}^L$ be the tenor-product B-spline basis of $V_h$. We know that the supports of tensor-product B-splines are axes-parallel boxes with bounded support ($|\supp \phi_\ell| \le c_1 h_{\mathbf Z}^d$) and bounded overlap ($\{k: \supp \phi_\ell \cap \supp \phi_k \ne \emptyset\} \le c_2$) for some constants $c_1$ and $c_2$ only depending on $p$. Moreover, we know that the basis functions are non-negative and form a partition of unity, thus $\phi_\ell \in [0,1]$. We use the basis to write
    $w_h = \sum_{\ell=1}^{L} w_\ell \phi_\ell$
    and
    $v_h = \sum_{\ell=1}^{L} v_\ell \phi_\ell$.
    For the choice
    $v_\ell := \overline{\rho}_\ell w_\ell$
    with
    $\overline \rho_\ell := \tfrac{1}{ |\supp \phi_\ell| } (\rho,1)_{L^2(\supp \phi_\ell)}$,
    we have
    \begin{align*}
            & \|\rho w_h - v_h\|_{L^2(\widehat \Omega)}^2
             =
            \left\|\sum_{\ell=1}^{L} ( \rho w_\ell - v_\ell) \phi_\ell \right\|_{L^2(\widehat \Omega)}^2
             \le (1+c_2^2) \sum_{\ell=1}^{L}
            \|  (\rho w_\ell - v_\ell) \phi_\ell \|_{L^2(\widehat \Omega)}^2 \\
            & \qquad \le (1+c_2^2) \sum_{\ell=1}^{L}
            \|  \rho w_\ell - v_\ell \|_{L^2(\supp \phi_\ell)}^2
            = (1+c_2^2) \sum_{\ell=1}^{L}
            |w_\ell|^2 \|  \rho - \overline \rho_\ell \|_{L^2(\supp \phi_\ell)}^2.
    \end{align*}
    Using a Poincar\'e inequality (cf. \cite{payne1960optimal}) with constant $c_3>0$, we obtain
    \begin{align*}
        \|\rho w_h - v_h\|_{L^2(\widehat \Omega)}^2
        & \le (1+c_2^2) c_3 h_{\mathbf Z}^2
        \sum_{\ell=1}^{L} |w_\ell|^2
        | \rho |_{H^1(\supp \phi_\ell)}^2
         \\&  \le c_1(1+c_2^2) c_3 h_{\mathbf Z}^{2+d}
        | \rho |_{W^1_\infty(\widehat \Omega)}^2
        \sum_{\ell=1}^{L} |w_\ell|^2
        .
    \end{align*}
    Theorem~\ref{thrm:scherer}  finishes the proof.
\end{proofof}

\bibliographystyle{abbrv}
\bibliography{bibliography}

\end{document}